\documentclass[12pt,a4paper, oneside]{amsart}

\usepackage{amsmath, nicefrac, amsthm, verbatim, amsfonts, amssymb, xcolor, bbm}
\usepackage{graphics, xspace, enumerate}
\usepackage{stix}
\usepackage[a4paper,margin=2.5cm]{geometry}

\usepackage[bb=boondox]{mathalfa}

\usepackage{graphicx}
\usepackage[colorlinks=true,citecolor=red,urlcolor=blue,linkcolor=red,bookmarksopen=true,unicode=true,pdffitwindow=true]{hyperref}
\usepackage[english]{babel}
\usepackage[languagenames,fixlanguage]{babelbib}
\hypersetup{pdfauthor={}}
\hypersetup{pdftitle={FLT for the capacity of the range of random walks}}

\theoremstyle{plain}
\newtheorem{theorem}{Theorem}[section]

\newtheorem{lemma}[theorem]{Lemma}

\theoremstyle{definition}

\newcommand {\R} {\ensuremath{\mathbb{R}}}
\newcommand {\ZZ} {\ensuremath{\mathbb{Z}}}

\newcommand{\Capa}{\operatorname{Cap}}
\newcommand{\Var}{\operatorname{Var}}

\newcommand{\bbN}{\mathbb{N}}
\newcommand{\bbZ}{\mathbb{Z}}
\newcommand{\bbR}{\mathbb{R}}

\newcommand{\bbP}{\mathbb{P}}
\newcommand{\bbE}{\mathbb{E}}

\newcommand{\calC}{\mathcal{C}}

\newcommand{\calE}{\mathcal{E}}
\newcommand{\calF}{\mathcal{F}}

\newcommand{\calD}{\mathcal{D}}

\newcommand{\calR}{\mathcal{R}}
\newcommand{\calN}{\mathcal{N}}

\newcommand{\aps}[1]{\vert #1 \vert}

\newcommand{\floor}[1]{\lfloor #1 \rfloor}

\numberwithin{equation}{section}

\title[FCLT for the range of stable random walks]{Functional CLT for the range of stable random walks}

\author[W.\ Cygan]{Wojciech Cygan}
\address[Wojciech Cygan]{Institut f\"{u}r Mathematische Stochastik\\Technische Universit\"{a}t Dresden\\Dresden\\Germany
\& 
Instytut Matematyczny\\Uniwersytet Wroc\l{}awski\\ Wroc\l{}aw\\ Poland}
\email{wojciech.cygan@uwr.edu.pl}

\author[N.\ Sandri\'{c}]{Nikola Sandri\'{c}}
\address[Nikola\ Sandri\'{c}]{Department of Mathematics\\University of Zagreb\\ Zagreb\\Croatia}
\email{nsandric@math.hr}

\author[S.\ \v{S}ebek]{Stjepan\ \v{S}ebek}
\address[Stjepan\ \v{S}ebek]{
	Institute of Discrete Mathematics\\
	Graz University of Technology\\
	Graz\\ 
	Austria
	\&	
		Department of Applied Mathematics\\
	Faculty of Electrical Engineering and Computing\\
	University of Zagreb\\ 
 Zagreb\\ 
	Croatia}
\email{stjepan.sebek@fer.hr}


\begin{document}
\allowdisplaybreaks[4]

\begin{abstract}
	In this note, we establish a functional central limit theorem for the capacity of the range for a class of  $\alpha$-stable random walks on the integer lattice $\mathbb{Z}^d$ with
	  $ d> 5\alpha/2$. Using similar methods, 
we also prove  an analogous result for the cardinality of the range when  $d > 3\alpha/ 2$.
\end{abstract}

\maketitle

\section{Introduction}

Let $(\Omega,\calF,\bbP)$ be a probability space, and let $\{\xi_i\}_{i\ge1}$ be a sequence of i.i.d.\ $\bbZ^d$-valued random variables defined on $(\Omega,\calF,\bbP)$, where $\bbZ^d$ denotes the $d$-dimensional integer lattice. Further, let $S_0=x$, and $S_n=S_{n-1}+\xi_n$, $n\ge1$, be a  $\bbZ^d$-valued random walk starting from $x\in\ZZ^d$. 
The range of the random walk $\{S_n\}_{n\ge0}$ is defined as the random set  
\begin{align*}
	\calR_n\, =\, \{S_0, \dots,S_n\}\,,\qquad n\geq 0\,.
	\end{align*} 
 Throughout the paper, we use the notation $|\calR_n|$ to denote 
the cardinality of $\calR_n$.

In addition to the cardinality of the range, we also consider its capacity. 
Let $\bbP_x$ be the probability measure (on the space $(\Omega,\mathcal{F})$) which corresponds to $\{S_n\}_{n\ge0}$ starting at $x\in\ZZ^d$. We write $\bbP$ instead of $\bbP_0$. 
For any $A\subseteq\ZZ^d$, we denote by $T^+_A$ the first hitting time of the set $A$ by $\{S_n\}_{n\ge0}$, that is,
$$
T^+_A\,=\,\inf\{n\ge1:S_n\in A\}\,.
$$
Also, when $A=\{x\}$ for $x\in\ZZ^d$, we write $T^+_x$ instead of $T^+_{\{x\}}$. Recall  $\{S_n\}_{n\ge0}$ is said to be transient if $\bbP (T^+_0=\infty)>0$; otherwise it is said to be recurrent, which implies that every random walk is either transient
or recurrent.
In the case when $\{S_n\}_{n\ge0}$  is transient,  the capacity of a set $A\subseteq\ZZ^d$ 
is defined as 
$$
{\rm Cap}\,(A)\,=\,\sum_{x\in A}\bbP_x(T^+_A=\infty)\,.
$$
For $n\ge0$, we denote $\calC_n={\rm Cap}\, (\calR_n)$. Observe that $\calC_n$ is a random variable.
The aim of this article is to prove a functional central limit theorem (FCLT) for the capacity and the cardinality of the range of the random walk  $\{S_n\}_{n\ge0}$, that is, for the  stochastic processes  $\{\calC_{\lfloor nt\rfloor}\}_{t\ge0}$ and $\{|\calR_{\lfloor nt\rfloor}|\}_{t\ge0}$.

The study on the range of random walks  in $\ZZ^d$ has a long history. A pioneering work is due to Dvoretzky and Erd\"{o}s \cite{Dvoretzky} where they obtained the strong law of large numbers for $\{|\calR_n|\}_{n\ge0}$ of a simple random walk in  $d\ge2$. 
This result was later extended by Spitzer  \cite{Spitzer} to all random walks  in $d\ge1$.  A central limit theorem for $\{|\calR_n|\}_{n\ge0}$ was first  obtained by Jain and Orey \cite{Jain-Orey} for strongly transient random walks (see below for the definition of strong transience). Jain and Pruitt \cite{Jain_Pruitt} later 
extended this result to all random walks in $d\ge3$. 
Le Gall \cite{LeGall-French} proved a version of a central limit theorem for $\{|\calR_n|\}_{n\ge0}$ of all two-dimensional random walks with zero mean and finite second moment. It is remarkable that in this case the limit law is not normal. 
For $d = 1$, Jain and Pruitt \cite[Theorem 6.1]{Jain_Pruitt_Berkeley} proved that $\bbE[|\calR_n|] \asymp \sqrt{n}$, where the symbol $\asymp$ means that the ratio of the two expressions is bounded from below and above by some positive constants. 
Le Gall and Rosen \cite{LeGall-Rosen} established the strong law of large numbers and  central limit theorem for $\{|\calR_n|\}_{n\ge0}$ of a class of $\alpha$-stable random walks. 

Studies on the long-time behavior  of  $\{\calC_n\}_{n\ge0}$ were initiated by  Jain and Orey in \cite{Jain-Orey} where they obtained a version of the strong law of large numbers for any transient random walk.  Asselah, Schapira and Sousi \cite{Asselah_Zd} proved  a central limit theorem  for $\{\calC_n\}_{n\ge0}$ of a simple random walk in $d\ge6$. 
Versions of the law of large numbers and central limit theorem in $d = 4$ were proved by the same authors in \cite{Asselah_Z4}, see also \cite{Chang}. 
Recently, Schapira \cite{Schapira} proved a central limit theorem for $\{\calC_n\}_{n\ge0}$ of a class of symmetric random walks in $\bbZ^5$ which satisfy appropriate moment conditions. 
In \cite{CSS19}, the present authors established a central limit theorem for $\{\calC_n\}_{n\ge0}$ of a class of $\alpha$-stable random walks in  $d > 5\alpha/2$.

A FCLT for $\{|\calR_n|\}_{n\ge0}$ was proved by Jain and Pruitt \cite{Jain_Pruitt_further} for all random walks in $d\ge3$  satisfying $\mathbb{P}(T^+_0=\infty)<1$  (note that if $\mathbb{P}(T^+_0=\infty)=1$ then $|\calR_n| =n+1$ a.s.). This result is a version of Donsker's invariance principle and it states that suitably normalized and linearly interpolated process $\{|\calR_n|\}_{n\ge0}$ converges weakly in the space of continuous functions endowed with the locally uniform topology to a standard one-dimensional Brownian motion.  
The purpose of the present article is to prove an analogous result for $\{\calC_{\lfloor nt\rfloor}\}_{t\ge0}$ and $\{|\calR_{\lfloor nt\rfloor}|\}_{t\ge0}$ of  a class of  $\alpha$-stable random walks. We remark that there is a correspondence between the limit behavior of the cardinality of the range when the ratio of the dimension to the index of stability is $\rho$ and the capacity of the range when this ratio is $\rho + 1$ (as already indicated in \cite{CSS19}).
Basing upon the results for the range when $d/\alpha \leq 3/2$ from \cite{LeGall-Rosen}, we expect that for $d/\alpha \leq 5/2$ the capacity of the range behaves in a different manner than presented in this article. 
If $d/\alpha = 5/2$, we conjecture that the limit law is again Gaussian but the scaling sequence should be of the form $\sqrt{ng(n)}$, where  $g(n)$ is a slowly varying function. This corresponds to the scaling sequence for the range process in the case $d/\alpha = 3/2$, as established in \cite[Section 4.5]{LeGall-Rosen}. We remark that the case $\alpha =2$ (and $d=5$) for the capacity process has been recently partially solved by Schapira  \cite{Schapira} where he studied a class of symmetric random walks which satisfy some moment condition and for this class he obtained the normal law in the limit while the scaling sequence was  $\sqrt{n\log n}$. 
For $2\leq d/\alpha<5/2$, we conjecture that the limit law is nonnormal and it is given in terms  of the  self-intersection local time of the limiting stable process (see  \cite[Theorem 1.2]{Asselah_Z4} for the case $\alpha=2$) and with the scaling sequence of the form $n^{3 - d/\alpha}g(n)$, where $g(n)$ is again a slowly varying function. 
In the context of the range process, we expect the same asymptotic behavior but  the corresponding scaling sequence should be of the form $n^{2 - d/\alpha}g(n)$, see \cite[Result 2]{LeGall-Rosen}.

We remark that there are further interesting results pertaining to the limit behavior of the process $\{|\calR_n|\}_{n\ge0}$, including the law of the iterated logarithm and an almost sure invariance principle, see \cite{Bass_Kumagai}, \cite{Bass_Rosen}, \cite{Hamana}, \cite{Jain_Pruitt_1970}, \cite{Jain_Pruitt_1970_2}, \cite{Jain_Pruitt_LIL}, and \cite{Jain_Pruitt_further}.

Before we state the main result of this article, we formulate and briefly discuss assumptions which we impose on the random walk  $\{S_n\}_{n\ge0}$.

\begin{itemize}
	\item[(\textbf{A1})] $\{S_n\}_{n\ge0}$ is   aperiodic, that is, the set $\{x\in\mathbb{Z}^d:\, \bbP(S_1=x)>0\}$ generates (as an additive subgroup) the whole of $\ZZ^d$.
	\item[(\textbf{A2})] $\{S_n\}_{n\ge0}$ belongs to the domain of  attraction of a  nondegenerate $\alpha$-stable law with index $0<\alpha\le2$. This means there exists a regularly varying function $b(x)$ with index $1/\alpha$ such that 
	$$\frac{S_n}{b(n)}\,\xrightarrow[n\nearrow\infty]{(\text{d})}\,X_\alpha\,,$$ 
where $X_\alpha$ is an  $\alpha$-stable random variable in $\R^d$ and $\xrightarrow[]{(\text{d})}$ stands for the convergence in distribution.
\item[(\textbf{A3})] $\{S_n\}_{n\ge0}$ is  symmetric and  strongly transient.
\item[(\textbf{A4})] $\{S_n\}_{n\ge0}$ admits one-step loops, that is, $\bbP (S_1=0)>0$.
\end{itemize}

We first remark that assumption (\textbf{A1}) is not restrictive. If $\{S_n\}_{n\ge0}$ were not aperiodic, we could perform our analysis (and obtain the same results) on the (smallest) additive subgroup generated by the set $\{x\in\mathbb{Z}^d:\, \bbP(S_1=x)>0\}$.

Assumption (\textbf{A2}) is of fundamental importance for our analysis. It allows us to apply error estimates in the capacity decomposition which we use in the proof of a FCLT for $\{\calC_{\lfloor nt\rfloor}\}_{t\ge0}$. Similarly, in view of (\textbf{A2}) we can apply results from \cite{LeGall-Rosen} to estimate the number of intersection points of two independent copies of our random walk which are necessary to prove a FCLT for  $\{|\calR_{\lfloor nt\rfloor}|\}_{t\ge0}$. 

To discuss assumption (\textbf{A3}), recall that $\{S_n\}_{n\ge0}$ is transient if $\bbP (T^+_0=\infty)>0$. 
Transience is equivalent to the convergence of the series $\sum_{n\geq 1}\bbP (S_n=0)$. 
The random walk $\{S_n\}_{n\ge0}$ is called strongly transient if
the series $\sum_{n\geq 1} \sum_{k\geq n} \bbP (S_k=0)$ converges.
We remark that in $d \ge 3$ every random walk is transient, and in $d \ge 5$ it is strongly transient. However, such random walks can also appear in lower dimensions.
In particular, under assumption (\textbf{A2}), $\{S_n\}_{n\ge0}$ is transient if $d>\alpha$ and strongly transient if $d>2\alpha$, see \cite[Theorem 3.4]{Sato} and \cite[Theorem 7]{Takeuchi}).  
The notion of strong transience was first introduced  in \cite{Port} for Markov chains and was later used in \cite{Jain-Orey} in the context of the limit behavior of the range of random walks.

Assumption  (\textbf{A4}) is purely technical. Accompanied by assumptions  (\textbf{A1})--(\textbf{A3}),  
it has recently enabled the present authors in \cite{CSS19} to conclude that the appropriately centered and  normalized stochastic process $\{\calC_n\}_{n \ge 0}$ converges weakly to a normal law. Notice that (\textbf{A4})  excludes a simple random walk. We show, however, that a FCLT for $\{\calC_{\lfloor nt\rfloor}\}_{t\ge0}$ and $\{|\calR_{\lfloor nt\rfloor}|\}_{t\ge0}$ of  this process holds true as well.

We now state the main results of the article.

\begin{theorem}\label{tm:FCLR_for_the_cap}
	Assume (\textbf{A1})-(\textbf{A4}), $0<\alpha\le2$, and $d/\alpha > 5/2$. 
Then, there is a constant $\sigma_d >0 $ such that 
\begin{equation*}
\left\{\frac{\calC_{\floor{nt}} - \bbE\bigl[\calC_{\floor{nt}}\bigr]}{\sigma_d \sqrt{n}}\right\}_{t\ge0}\, 
\xrightarrow[n\nearrow\infty]{(\text{J}_1)}\,
\{B_t\}_{t\ge0}\,,
\end{equation*}
where $\lfloor x\rfloor$ denotes the integer part of $x\in\R$, $\ \xrightarrow[n\nearrow\infty]{(\text{J}_1)}$ stands for the weak convergence in 
the Skorohod space $\calD([0, \infty), \bbR)$ endowed with the $\text{J}_1$ topology, and $\{B_t\}_{t\geq 0}$ denotes a standard one-dimensional Brownian motion.
\end{theorem}

Our proof is based on the central limit theorem from \cite{CSS19} and a tightness argument. To establish tightness, we utilize the capacity decomposition from \cite[Corollary 2.1]{Asselah_Zd} and combine it with the error term estimates from \cite[Lemma 3.1]{CSS19} which hold under crucial assumptions (\textbf{A1})--(\textbf{A3}). Using a similar reasoning and estimates of the number of intersection points (extracted from \cite{LeGall-Rosen}), we  prove an analogous result for   $\{|\calR_{\lfloor nt\rfloor}|\}_{t\ge0}$.
\begin{theorem}\label{tm:FCLR_for_the_car}
	Assume (\textbf{A1}), (\textbf{A2}), $0<\alpha\le2$, $d/\alpha > 3/2$ and $\bbP(T^+_0=\infty)<1$.
	Then, there is a constant $\sigma_d >0 $ such that 
	\begin{equation*}
	\left\{\frac{|\calR_{\floor{nt}}| - \bbE\bigl[|\calR_{\floor{nt}}|\bigr]}{\sigma_d \sqrt{n}}\right\}_{t\ge0}\, 
	\xrightarrow[n\nearrow\infty]{(\text{J}_1)}\,
	\{B_t\}_{t\ge0}\,.
	\end{equation*}
	where  we use the same notation as in Theorem \ref{tm:FCLR_for_the_cap}.
\end{theorem}

 This result has to be compared with the Donsker's invariance principle for $\{|\calR_n|\}_{n\ge0}$ which was found by Jain and Pruitt \cite{Jain_Pruitt_further}. They proved that performing an appropriate linearization of $\{|\calR_n|\}_{n\ge0}$, the uniform convergence towards a Brownian motion holds in two cases: (i) for all random walks in dimensions $d\geq 3$ which satisfy $\bbP(T^+_0=\infty)<1$, and (ii) for all strongly transient random walks for which a certain moment condition is valid. We notice that this condition is usually not easy to check, see the closing Remark in \cite{Jain_Pruitt_further}.
We establish a slightly more general result, as we prove convergence in the Skorohod space, but for stable random walks only. Notice, however, that we may well have $d<3$.
We also mention that we could easily deduce the locally uniform convergence in the space of continuous functions by employing a suitable linearization.

\subsection*{Acknowledgement}
We are grateful to Prof.\ Alexander M.\ Iksanov (National Taras Shevchenko University of Kyiv) for having drawn our attention to the topic of this paper during the conference \textit{Probability and Analysis 2019} in B\k{e}dlewo, Poland. 

This work has been supported by \textit{Deutscher Akademischer Austauschdienst} (DAAD) and \textit{Ministry of Science and Education of the Republic of Croatia} (MSE) via project \textit{Random Time-Change and Jump Processes}. Financial support through the \textit{Alexander von Humboldt Foundation} and \textit{Croatian Science Foundation} under projects 8958 and 4197 (for N.\ Sandri\'c), and the \textit{Austrian Science Fund} (FWF) under project P31889-N35 and \textit{Croatian Science Foundation} under project 4197 (for S.\ \v Sebek) is gratefully acknowledged.

We also thank the anonymous referees for helpful comments that
have led to improvements to the presentation of the article.


\section{FCLT for the process $\{\calC_{\lfloor nt\rfloor}\}_{t\ge0}$}\label{sec:capacity}
Throughout this section, we assume that  $\{S_n\}_{n\geq0}$ satisfies (\textbf{A1})--(\textbf{A4}). We denote by $G(x,y) $ the corresponding Green function, that is, 
\begin{align*}
G(x,y) \,=\, \sum_{n=0}^\infty \bbP_x (S_n=y)\,, \qquad x, y \in \bbZ^d\,.
\end{align*}
Also, for $A,B\subseteq\ZZ^d$ we denote
\begin{align*}
G(A, B) \,=\, \sum_{x\in A}\sum_{y\in B} G(x,y)\,.
\end{align*}

In order to prove a FCLT for $\{\calC_{\lfloor nt\rfloor}\}_{t\ge0}$, we employ the following classical two-step scheme, see \cite[Theorems 16.10 and  16.11]{Kallenberg}.
Let $\{X^{n}\}_{n \ge 0}$ be a sequence of random elements in the Skorohod space $\calD([0, \infty), \bbR)$ endowed with the $\text{J}_1$ topology. The sequence $\{X^{n}\}_{n \ge 0}$ converges weakly to a random element $X$ (in $\calD([0, \infty), \bbR)$) if the following two conditions are satisfied: 
\begin{enumerate}[(i)]
	\item The finite-dimensional distributions of $\{X^{n}\}_{n \ge 0}$ converge weakly to the finite-dimensional distributions of $X$.
	\item For any bounded sequence $\{T_n\}_{n \ge1}$ of $\{X^{n}\}_{n \ge 0}$-stopping times,  any sequence $\{h_n\}_{n \ge 1}\subset[0,\infty)$  converging to zero, and any $\varepsilon > 0$, it holds that
		\begin{equation*}
		\lim_{n \nearrow \infty}\bbP \bigl(|X_{T_n + h_n}^{n} - X_{T_n}^{n}| \ge \varepsilon\bigr)\, =\, 0\,.
	\end{equation*}
\end{enumerate}

\begin{proof}[Proof of Theorem \ref{tm:FCLR_for_the_cap}] 
We consider the following sequence of random elements which are defined in the space $\calD([0, \infty), \bbR)$,
\begin{equation}\label{eq:def_of_Y^n_t}
	X_t^{n} \,= \,\frac{\calC_{\floor{nt}} - \bbE\bigl[\calC_{\floor{nt}}\bigr]}{\sigma_d \sqrt{n}}\,,\qquad  n\ge1\,, 
\end{equation}
where $\sigma_d$ is a positive constant. We  prove validity of conditions (i) and (ii) under assumptions (\textbf{A1})--(\textbf{A4}). Let us start by showing condition (i).
\smallskip

\noindent 
\textit{Condition (i)}.
By \cite[Theorem 1.1.]{CSS19}, we have that for any $t>0$
\begin{equation}\label{eq:CLT-Y_t^n}
	X_t^{n}\, =\, \frac{\calC_{\floor{nt}} - \bbE\bigl[\calC_{\floor{nt}}\bigr]}{\sigma_d \sqrt{\floor{nt}}} \cdot \frac{\sqrt{\floor{nt}}}{\sqrt{nt}} \cdot \frac{\sqrt{nt}}{\sqrt{n}}\, \xrightarrow[n\nearrow\infty]{(\text{d})}\, \calN(0, t)\,,
\end{equation}
where\ \   $\xrightarrow[n\nearrow\infty]{(\text{d})}$ stands for the convergence
in distribution, the constant $\sigma_d>0$ is determined in \cite[Theorem 1.1.]{CSS19} and
$\calN(0, t)$ stands for a Gaussian random variable  with mean zero and variance $t$.
Let $k \ge1$ be an arbitrary integer and choose $0 = t_0 < t_1 < t_2 < \cdots < t_k $. We need to prove that
\begin{equation*}
	(X_{t_1}^{n}, X_{t_2}^{n}, \ldots, X_{t_k}^{n})\, \xrightarrow[n\nearrow\infty]{(\text{d})}\, (B_{t_1}, B_{t_2}, \ldots, B_{t_k})\,.
\end{equation*}
In view of the Cram\'{e}r-Wold theorem \cite[Corollary 5.5]{Kallenberg}, it suffices to show that
\begin{equation}\label{eq:cvg_CW_cap}
	\sum_{j = 1}^k \nu_j X_{t_j}^{n}\, \xrightarrow[n\nearrow\infty]{(\text{d})}\, \sum_{j = 1}^k \nu_j B_{t_j}\,, \qquad  (\nu_1, \nu_2, \ldots, \nu_k) \in \bbR^k\,.
\end{equation}
To prove \eqref{eq:cvg_CW_cap}, we  first find lower and upper bounds for $\calC_{\floor{nt_j}}$, $j=1,\dots,k$.
We  follow closely the arguments from the  capacity decomposition \cite[Corollary 2.1]{Asselah_Zd}. For  $i =1,  \ldots, k,$ we have that
\begin{align*}
	\calC_{\floor{nt_i}}
	& \, =\, \Capa\,\bigl(\calR_{\floor{nt_1}} \cup \calR[\floor{nt_1}, \floor{nt_i} ] \bigr)  \\
	& \,=\, \Capa\,\bigl( \bigl(\calR_{\floor{nt_1}} - S_{\floor{nt_1}}\bigr) \cup \bigl(\calR[ \floor{nt_1}, \floor{nt_i}] - S_{\floor{nt_1}}\bigr) \bigr)\,,
	\end{align*}
	where for $ 1\le m\le n$, $\calR[m,n]:=\{S_m,\dots,S_n\}$. 
By the Markov property, the two random variables 
\begin{align*}
\calR^{(1)}_{\floor{nt_1}}\, :=\, \calR_{\floor{nt_1}} - S_{\floor{nt_1}}\qquad \text{and}\qquad
\calR^{(2)}_{\floor{nt_i} - \floor{nt_1}} \,:=\, \calR[ \floor{nt_1}, \floor{nt_i}] - S_{\floor{nt_1}}
\end{align*}
are independent, and $\calR^{(2)}_{\floor{nt_i} - \floor{nt_1}} $ has the same law as $\calR_{\floor{nt_i} - \floor{nt_1}} $. The symmetry of the random walk $\{S_n\}_{n \ge 0}$ implies that $\calR^{(1)}_{\floor{nt_1}} $ is equal in law to $\calR_{\floor{nt_1}} $. Hence, \cite[Proposition 1.2]{Asselah_Zd} implies
\begin{align*}	
\calC_{\floor{nt_i}}
	& \,\ge\, 
	\Capa\,\bigl(\calR^{(1)}_{\floor{nt_1}}\bigr) + \Capa\,\bigl(\calR^{(2)}_{\floor{nt_i} - \floor{nt_1}}\bigr) 
	- 2G\bigl(\calR^{(1)}_{\floor{nt_1}}, \calR^{(2)}_{\floor{nt_i} - \floor{nt_1}}\bigr)\,.
\end{align*}
We now present how to deal with the next step of the decomposition. We have
\begin{align*}
\Capa\,\bigl(\calR^{(2)}_{\floor{nt_i} - \floor{nt_1}}\bigr)
&\,=\,
\Capa \,\bigl(\calR[ \floor{nt_1}, \floor{nt_i}] - S_{\floor{nt_1}})\bigr)\\
&\,=\,\Capa\, \bigl( \calR[ \floor{nt_1}, \floor{nt_i}] \bigr)\\
&\,=\,\Capa\,\bigl( \bigl(\calR[ \floor{nt_1}, \floor{nt_2}] - S_{\floor{nt_2}}\bigr) \cup \bigl(\calR[ \floor{nt_2}, \floor{nt_i}] - S_{\floor{nt_2}}\bigr) \bigr)\,.
\end{align*}
Similarly as before, the two random variables
\begin{align*}
\calR^{(2)}_{\floor{nt_2}-\floor{nt_1}}\, :=\, \calR[ \floor{nt_1}, \floor{nt_2}] - S_{\floor{nt_2}}
\qquad \text{and}\qquad
\calR^{(3)}_{\floor{nt_i}-\floor{nt_2}} \,:=\, \calR[ \floor{nt_2}, \floor{nt_i}] - S_{\floor{nt_2}}
\end{align*}
are independent. Also, the random variable $\calR^{(2)}_{\floor{nt_2}-\floor{nt_1}} $ has the same law as
$\calR_{\floor{nt_2}-\floor{nt_1}} $, and $\calR^{(3)}_{\floor{nt_i}-\floor{nt_2}}$ has the same law as $\calR_{\floor{nt_i}-\floor{nt_2}}$. If we continue with this procedure and use the same arguments as above, together with subadditivity property of the capacity (see \cite[Proposition 25.11]{Spitzer}) for the upper bound, we  obtain the following estimates
\begin{equation*}\label{eq:LB_for_Cap_ntj}
	  \sum_{i = 1}^j \calC^{(i)}_{\floor{nt_i} - \floor{nt_{i - 1}}} - 2\sum_{i = 1}^{j - 1}\calE^{(i)}_{\floor{nt_j}}
	  \,\leq \,
	  \calC_{\floor{nt_j}}
	 \, \leq \,
	  \sum_{i = 1}^j \calC^{(i)}_{\floor{nt_i} - \floor{nt_{i - 1}}}\,,\qquad j=1,\ldots ,k\,,
\end{equation*}
where 
\begin{equation*}
\calC^{(i)}_{\floor{nt_i} - \floor{nt_{i - 1}}}\, :=\, \Capa\,\bigl(\calR^{(i)}_{\floor{nt_i} - \floor{nt_{i - 1}}}\bigr)
\qquad \text{and}\qquad
	\calE^{(i)}_{\floor{nt_j}} \,:=\, G\bigl(\calR^{(i)}_{\floor{nt_i}-\floor{nt_{i-1}}}, \calR^{(i + 1)}_{\floor{nt_{j}}-\floor{nt_{i}}}\bigr)\,.
\end{equation*}
The random variables $\calC^{(i)}_{\floor{nt_i} - \floor{nt_{i - 1}}}$, $i=1,\dots,k,$ are independent, $\calR^{(i)}_{\floor{nt_i} - \floor{nt_{i - 1}}}$ has the same law as $\calR_{\floor{nt_i} - \floor{nt_{i - 1}}}$ and 
$	\calE^{(i)}_{\floor{nt_j}}$ has the same law as 
$G(\calR_{\floor{nt_i}-\floor{nt_{i-1}}}, \widetilde{\calR}_{\floor{nt_{j}}-\floor{nt_{i}}})$, with $\widetilde{\calR}_{\floor{nt_{j}}-\floor{nt_{i}}}$ being an independent copy of $\calR_{\floor{nt_{j}}-\floor{nt_{i}}}$.  

We now find lower and upper bounds for the left-hand side expression in \eqref{eq:cvg_CW_cap}.
For $\nu_j \ge 0$, we have
\begin{equation*}
	\nu_j \sum_{i = 1}^j \calC^{(i)}_{\floor{nt_i} - \floor{nt_{i - 1}}} - 
	2 \nu_j \sum_{i = 1}^{j - 1} \calE^{(i)}_{\floor{nt_j}}
	\,\le\, \nu_j \calC_{\floor{nt_j}} 
	\,\le\, \nu_j \sum_{i = 1}^j \calC^{(i)}_{\floor{nt_i} - \floor{nt_{i - 1}}}\,,
\end{equation*}
and for $\nu_j< 0$,
\begin{equation*}
		\nu_j \sum_{i = 1}^j \calC^{(i)}_{\floor{nt_i} - \floor{nt_{i - 1}}} - 
	2 \nu_j \sum_{i = 1}^{j - 1} \calE^{(i)}_{\floor{nt_j}}
	\,\ge\, \nu_j \calC_{\floor{nt_j}} 
	\,\ge\, \nu_j \sum_{i = 1}^j \calC^{(i)}_{\floor{nt_i} - \floor{nt_{i - 1}}}\,.
\end{equation*}
Thus, by splitting the sum into two parts, we obtain the following lower bound
\begin{align}\label{eq:2error terms}
	\sum_{j = 1}^k \nu_j X_{t_j}^{n}
	&  \,\ge\, 
	\sum_{\substack{1\leq j\leq k \\\nu_j\geq 0}} \frac{1}{\sigma_d \sqrt{n}} \left(\nu_j \sum_{i = 1}^j 
	\bigl(\calC^{(i)}_{\floor{nt_i} - \floor{nt_{i - 1}}} - 
		\bbE\bigl[\calC^{(i)}_{\floor{nt_i} - \floor{nt_{i - 1}}} \bigr]\bigr) 
-2\nu_j 	 \sum_{i = 1}^{j - 1} \calE^{(i)}_{\floor{nt_j}}\right) \nonumber\\
&\qquad + 
\sum_{\substack{1\leq j\leq k \\\nu_j< 0}} \frac{1}{\sigma_d \sqrt{n}} \left(\nu_j \sum_{i = 1}^j 
	\bigl(\calC^{(i)}_{\floor{nt_i} - \floor{nt_{i - 1}}} - 
		\bbE\bigl[\calC^{(i)}_{\floor{nt_i} - \floor{nt_{i - 1}}} \bigr]\bigr) 
+2\nu_j 	 \sum_{i = 1}^{j - 1} \bbE\bigl[\calE^{(i)}_{\floor{nt_j}}\bigr]\right)\nonumber \\
	& \,=\, \sum_{i = 1}^k \left(\sum_{j = i}^k \nu_j\right) J^{(i)}_n 
	- \frac{2}{\sigma_d \sqrt{n}} \sum_{\substack{1\leq j\leq k \\\nu_j\geq 0}}
	 \nu_j \sum_{i = 1}^{j - 1} \calE^{(i)}_{\floor{nt_j}}
	 +
	 \frac{2}{\sigma_d \sqrt{n}} \sum_{\substack{1\leq j\leq k \\\nu_j< 0}}
	 \nu_j \sum_{i = 1}^{j - 1} \bbE \bigl[ \calE^{(i)}_{\floor{nt_j}}\bigr]\,,
\end{align}
where
\begin{align*}
J^{(i)}_n \,=\,\frac{\calC^{(i)}_{\floor{nt_i} - \floor{nt_{i - 1}}} - 
					\bbE\bigl[ \calC^{(i)}_{\floor{nt_i} - \floor{nt_{i - 1}}}  \bigr]}{\sigma_d \sqrt{n}}\,, \qquad i=1,\dots,k\,.
\end{align*}
We now study weak convergence  of $\{J_n^{(i)}\}_{n\ge1}$ and for this we can replace $\{\calC_n^{(i)}\}_{n\ge0}$ by $\{\calC_n\}_{n\ge0}$.
We clearly have
\begin{align}\label{eq:trivial}
\floor{x-y}\,\leq\, \floor{x}-\floor{y}\,\leq\, \floor{x-y}+1\,,\qquad  x\geq y\geq 0\,.
\end{align}
Also, the map $n\mapsto \calC_n$ is monotone, and it holds that
\begin{equation}\label{eq:C_n+1_C_n}
	\calC_{n + 1}\, =\, \Capa\,(\calR_{n + 1}) \,=\, \Capa\,(\calR_n \cup \{S_{n + 1}\}) \,\le\, \Capa\,(\calR_n) + \Capa\,(S_{n + 1})\, \le\, \calC_n + 1\,.
\end{equation}
Combining \eqref{eq:trivial} and \eqref{eq:C_n+1_C_n} with  \eqref{eq:CLT-Y_t^n}, 
we easily conclude that
\begin{equation*}
	J_n^{(i)}\, \xrightarrow[n\nearrow\infty]{(\text{d})}\, \calN(0, t_i - t_{i - 1})\,, \qquad i=1, 2, \ldots, k\,.
\end{equation*}
Next, we show that the two last terms in \eqref{eq:2error terms} are negligible. Indeed, 
the Markov inequality together with \cite[Lemma 3.2]{CSS19} implies that there is a constant $C>0$ such that for every $\varepsilon > 0$, 
\begin{equation}\label{eq:bound_with_Hd}
	\bbP\bigl(n^{-1/2}\calE^{(i)}_{\floor{nt_j}} > \varepsilon \bigr) 
\,	\le\, \frac{\bbE\bigl[\calE^{(i)}_{\floor{nt_j}}\bigr]}{\varepsilon \sqrt{n}}
	\,\le\, \frac{\bbE\bigl[ G\bigl(\calR_{\floor{nt_j}}, \widetilde{\calR}_{\floor{nt_j}}\bigr)\bigr]}{\varepsilon \sqrt{n}}
	 \,\le\, \frac{C H_d(\floor{nt_j})}{\varepsilon \sqrt{n}},
\end{equation}
where
	\begin{equation}\label{eq:def_of_Hd}
			H_d(n) \,=\, 
	\begin{cases}
		1\,, &\quad d/\alpha  > 3\,,\\
		\sum_{k=1}^n k^{-1}\ell (k)^{-d}\,, & \quad d/\alpha = 3\,,\\
		n^3 (b(n))^{-d}\,, & \quad 2 < d/\alpha < 3\,,\\
	\end{cases}	
	\end{equation}
and the function $b(x)$ is necessarily of the form
\begin{equation}\label{eq:shape_of_function_b}
	b(x) \,=\, x^{1/\alpha} \ell(x),\qquad x\geq 0\,,
\end{equation}
where $\ell(x)$ is a slowly varying function, see \cite{BGT_book}. Moreover, by \cite[Lemma 2.2]{LeGall-Rosen}, the function $n \mapsto \sum_{k=1}^n k^{-1}\ell (k)^{-d}$ is slowly varying. Hence, $H_d(n)$ is a slowly varying function in the case $d/\alpha \ge 3$ and regularly varying function with index strictly smaller than $1/2$ when $d/\alpha \in (5/2, 3)$. This implies that the last term in \eqref{eq:bound_with_Hd} tends to zero as $n \rightarrow \infty$.
Recall that the random variables $\calC^{(i)}_{\floor{nt_i} - \floor{nt_{i - 1}}}$, $i=1,\ldots ,k$, are independent. Therefore, after performing the same analysis as in \eqref{eq:2error terms} for the upper bound,
we conclude that 
\begin{equation*}
	\sum_{j = 1}^k \nu_j X_{t_j}^{n}\, \xrightarrow[n\nearrow\infty]{(\text{d})}\, 
	\calN\left(0, \sum_{i = 1}^k \Bigl(\sum_{j = i}^k \nu_j\Bigr)^2 (t_i - t_{i - 1})\right).
\end{equation*}
We finally notice that 
\begin{equation*}
	\sum_{i = 1}^k \Bigl(\sum_{j = i}^k \nu_j\Bigr)^2 (t_i - t_{i - 1})\, =\, \begin{bmatrix}
		\nu_1 \\
		\nu_2 \\
		\nu_3 \\
		\vdots \\
		\nu_k
	\end{bmatrix}^{\, \mathrm{t}}
	\begin{bmatrix}
		t_1 & t_1 & t_1 & \ldots & t_1 \\
		t_1 & t_2 & t_2 & \ldots & t_2 \\
		t_1 & t_2 & t_3 & \ldots & t_3 \\
		\vdots & \vdots & \vdots & \ddots & \vdots \\
		t_1 & t_2 & t_3 & \ldots & t_k
	\end{bmatrix}
	\begin{bmatrix}
		\nu_1 \\
		\nu_2 \\
		\nu_3 \\
		\vdots \\
		\nu_k
	\end{bmatrix}
	\,= \,\sum_{i = 1}^k \bigl(\nu_i^2 t_i + 2 \sum_{j > i} \nu_i \nu_j t_i\bigr)\,.
\end{equation*}
Hence, the finite-dimensional distributions of $\{X^{n}\}_{n\ge1}$ converge weakly to the finite-dimensional distributions of a one-dimensional  standard Brownian motion.
\smallskip

\noindent 
\textit{Condition (ii)}. Let $\{T_n\}_{n \ge1}$ be a bounded sequence of $\{X^{n}\}_{n\ge1}$-stopping times and $\{h_n\}_{n \ge1}\subset[0,\infty)$ an arbitrary sequence  which converges to zero. We want to prove that
\begin{equation}\label{ToShow}
	X^{n}_{T_n + h_n} - X^{n}_{T_n}\, \xrightarrow[n\nearrow\infty]{\bbP}\, 0\,,
\end{equation}
where $\,\,\xrightarrow[n\nearrow\infty]{\bbP}\,$ stands for the convergence in probability.
By \eqref{eq:def_of_Y^n_t}, we have
\begin{equation*}
	X^{n}_{T_n + h_n} - X^{n}_{T_n} \,=\, \frac{\calC_{\floor{nT_n + nh_n}} - \bbE\bigr[\calC_{\floor{nT_n + nh_n}}\bigl]}{\sigma_d \sqrt{n}} - \frac{\calC_{\floor{nT_n}} - \bbE\bigl[\calC_{\floor{nT_n}}\bigr]}{\sigma_d \sqrt{n}}\,.
\end{equation*}
Proceeding as in  the  capacity decomposition from \cite[Corollary 2.1]{Asselah_Zd} and combining  the strong Markov property with the subadditivity and monotonicity of the capacity yield
\begin{align*}
\calC_{\floor{n(T_n + h_n)}} 
	\,\le\, \calC_{\floor{nT_n} + \floor{nh_n}+1}
	\,\le\, \calC^{(1)}_{\floor{nT_n}} + \calC^{(2)}_{\floor{nh_n}+1  }\,,
\end{align*}
and
\begin{equation*}
\calC_{\floor{n(T_n + h_n)}} 
\,\geq\, 
 \calC_{\floor{nTn}+\floor{nh_n}}
 \,\geq\, 
	\calC^{(1)}_{\floor{nT_n}} 
	+ \calC^{(2)}_{\floor{nh_n}} 
	- 2\calE\bigl(\floor{nT_n}, \floor{nh_n}\bigr) \,,
\end{equation*}
where $\calC^{(1)}_{\floor{nT_n}}$ and $\calC^{(2)}_{\floor{nh_n}}$  are independent and  have the same law as $\calC_{\floor{nT_n}}$ and $\calC_{\floor{nh_n}}$, respectively. Moreover, the random variable $\calE(\floor{nT_n}, \floor{nh_n})  $ has the same law as $G(\calR_{\floor{nT_n}}, \widetilde{\calR}_{\floor{nh_n}})$, with $\widetilde{\calR}_{\floor{nh_n}}$ being an independent copy of $\calR_{\floor{nh_n}}$.

Using these inequalities, we now bound the expression $X^{n}_{T_n + h_n} - X^{n}_{T_n}$ from below and above with quantities converging to zero in probability.
We start with the lower bound
\begin{align*}
	X^{n}_{T_n + h_n} - X^{n}_{T_n}
	& \,\ge\,  \frac{\calC^{(2)}_{\floor{nh_h}} - \bbE\bigl[\calC_{\floor{nh_h}+1}\bigr]}{\sigma_d \sqrt{n}} - \frac{2\calE\bigl(\floor{nT_n}, \floor{nh_n}\bigr)  }{\sigma_d \sqrt{n}}\\
	& \,=\, \frac{\calC^{(2)}_{\floor{nh_h}} - \bbE\bigl[\calC_{\floor{nh_h}}\bigr]}{\sigma_d \sqrt{n}} - \frac{2\calE\bigl(\floor{nT_n}, \floor{nh_n}\bigr)  }{\sigma_d \sqrt{n}} + \frac{\bbE\bigl[\calC_{\floor{nh_h}}\bigr] - \bbE\bigl[\calC_{\floor{nh_h} + 1}\bigr]}{\sigma_d \sqrt{n}} \\
	& \,\ge\, \frac{\calC^{(2)}_{\floor{nh_h}} - \bbE\bigl[\calC_{\floor{nh_h}}\bigr]}{\sigma_d \sqrt{n}} - \frac{2\calE\bigl(\floor{nT_n}, \floor{nh_n}\bigr)  }{\sigma_d \sqrt{n}} - \frac{1}{\sigma_d \sqrt{n}}\,,
\end{align*}
where in the last line we used \eqref{eq:C_n+1_C_n}.
It suffices to show that
\begin{equation}\label{eq:LB_goes_to_0_in_P_cap}
	\frac{\calC^{(2)}_{\floor{nh_h}} - \bbE\bigl[\calC_{\floor{nh_h}}\bigr]}{\sigma_d \sqrt{n}} - \frac{2\calE\bigl(\floor{nT_n}, \floor{nh_n}\bigr)  }{\sigma_d \sqrt{n}} \,\xrightarrow[n \nearrow \infty]{\bbP}\, 0\,.
\end{equation}
Take arbitrary $\varepsilon > 0$. The  Markov inequality together with the fact that there is a constant $C>0$ such that $\sup_{n\ge1}\max\{T_n,h_n\}\leq C$ implies
\begin{align*}
	\bbP
	& \left( n^{-1/2} \left| \bigl(\calC^{(2)}_{\floor{nh_h}} - \bbE\bigl[\calC_{\floor{nh_h}}\bigr] \bigr) - 2\calE\bigl(\floor{nT_n}, \floor{nh_n}\bigr)  \right| > \sigma_d \varepsilon\right) \\
	& \,\le\, \bbP \left(2\bigl|\calC^{(2)}_{\floor{nh_n}} - \bbE\bigl[\calC_{\floor{nh_n}}\bigr]\bigr| > \varepsilon \sigma_d \sqrt{n}\right) + \bbP \left(4 \calE\bigl(\floor{nT_n}, \floor{nh_n}\bigr)   > \varepsilon \sigma_d \sqrt{n}\right) \\
	& \,\le\, \frac{4 \Var\bigl(\calC_{\floor{nh_n}}\bigr)}{\varepsilon^2 \sigma_d^2 n} + \frac{4 \bbE\bigl[G\bigl(\calR_{\floor{nT_n}}, \widetilde{\calR}_{\floor{nh_n}}\bigr)\bigr]}{\varepsilon \sigma_d \sqrt{n}}\\
&	\,\le\, \frac{4C_1 n h_n}{\varepsilon^2 \sigma_d^2 n} + \frac{4C_2 H_d(Cn)}{\varepsilon \sigma_d \sqrt{n}}\, \xrightarrow[n \nearrow \infty]{} \,0\,,
\end{align*}
where in the last line we applied  \cite[Lemma 4.3]{CSS19} to conclude that there is a constant $C_1>0$ such that $\Var(\calC_n) \le C_1 n$ for all $n \ge1$, and  \cite[Lemma 3.2]{CSS19} to find a constant $C_2>0$ such that $\bbE[G(\calR_{\floor{nT_n}}, \widetilde{\calR}_{\floor{nh_n}})] \le C_2 H_d(Cn)$ for all $n \ge1$, where $H_d(n)$ is defined in \eqref{eq:def_of_Hd}. Its index of regular variation is strictly smaller than $1/2$ for all $d/\alpha > 5/2$. This gives us  \eqref{eq:LB_goes_to_0_in_P_cap}.

To obtain the upper bound, we write
\begin{align*}
	X^{n}_{T_n + h_n} - X^{n}_{T_n}
	& \,\le\, \frac{\calC^{(2)}_{\floor{nh_h} + 1} - \bbE\bigl[\calC_{\floor{nh_h}}\bigr]}{\sigma_d \sqrt{n}} + \frac{2\bbE\bigl[G\bigl(\calR_{\floor{nT_n}}, \widetilde{\calR}_{\floor{nh_n}}\bigr)\bigr]}{\sigma_d \sqrt{n}} \\
	&\, =\, \frac{\calC^{(2)}_{\floor{nh_h}} - \bbE\bigl[\calC_{\floor{nh_h}}\bigr]}{\sigma_d \sqrt{n}} + \frac{2\bbE\bigl[G\bigl(\calR_{\floor{nT_n}}, \widetilde{\calR}_{\floor{nh_n}}\bigr)\bigr]}{\sigma_d \sqrt{n}} + \frac{\calC^{(2)}_{\floor{nh_h} + 1} - \calC^{(2)}_{\floor{nh_h}}}{\sigma_d \sqrt{n}} \,.
\end{align*}
One can easily show that in view of \eqref{eq:C_n+1_C_n} the last term converges to zero in law (and whence in probability). 
Finally,  for the remaining terms, we use the same arguments as before, which allow us to conclude \eqref{ToShow} and the proof is finished. 
\end{proof}

The next theorem is the corresponding FCLT for a simple random walk. Recall that a simple random walk does not satisfy (\textbf{A4}).
\begin{theorem}\label{rem:FCLT_for_simple_RW}
Let $\{S_n\}_{n\geq 0}$ be a symmetric simple  random walk in $\mathbb{Z}^d$ with $d\geq 6$. Then, the following convergence holds
\begin{equation*}
\left\{\frac{\calC_{\floor{nt}} - \bbE\bigl[\calC_{\floor{nt}}\bigr]}{\sigma_d \sqrt{n}}\right\}_{t\ge0} \,
\xrightarrow[n\nearrow\infty]{(\text{J}_1)}\,
\{B_t\}_{t\ge0}\,,
\end{equation*}
where $\sigma_d>0$ is the constant from \cite[Theorem 1.1]{Asselah_Zd}.
\end{theorem}
\begin{proof}
To prove the theorem, we use the same reasoning as in the proof of Theorem \ref{tm:FCLR_for_the_cap} but we use results from \cite{Asselah_Zd} instead of results from \cite{CSS19}. More precisely, instead of \cite[Lemma 3.2]{CSS19}, we use \cite[Lemma 3.2]{Asselah_Zd}, \cite[Lemma 4.3]{CSS19} has to be replaced with \cite[Lemma 3.5]{Asselah_Zd}, and, finally, \cite[Theorem 1.1]{CSS19} is replaced with \cite[Theorem 1.1]{Asselah_Zd}.
\end{proof}

\section{FCLT for the process  $\{|\calR_{\lfloor nt\rfloor}|\}_{t\ge0}$}\label{sec:cardinality}
In this section, we show how  to adapt the methods from the previous section and prove Theorem \ref{tm:FCLR_for_the_car}. 
Recall that we assume that $\{S_n\}_{n\ge0}$ satisfies (\textbf{A1}), (\textbf{A2}), $d/\alpha>3/2$ and $\bbP(T^+_0=\infty)<1$ (if $\bbP(T^+_0=\infty)=1$, then $|\calR_n|=n+1$ a.s.). Also, notice that by \cite[Proposition 2.4]{LeGall-Rosen} if $\alpha >1$, then $\bbP(T^+_0=\infty)<1$ holds true.

Before we prove Theorem \ref{tm:FCLR_for_the_car}, we formulate the following lemma which shows how the range of a random walk up to time $n$ can be decomposed into two independent ranges. This idea was first used by Le Gall in \cite{LeGall-French} to establish a central limit theorem for $\{|\calR_n|\}_{n\ge0}$.
Let $\{S_n\}_{n\ge0}$ and $\{\widetilde S_n\}_{n\ge0}$ be two independent and identically distributed random walks  (defined on the same probability space). By $I_n$, we denote the number of intersection points up to time $n$ of the paths of $\{S_n\}_{n\ge0}$ and $\{\widetilde S_n\}_{n\ge0}$, that is,
\begin{align*}
I_{ n}\,=\, \aps{\calR_{ n} \cap \widetilde{\calR}_{ n}}\,,
\end{align*}
where $\{\calR_n\}_{n\ge0}$ is the range of $\{S_n\}_{n\ge0}$, and $\{\widetilde{\calR}_n\}_{n\ge0}$ is the range of $\{\widetilde S_n\}_{n\ge0}$.
\begin{lemma}\label{lm:cardinality_decomposition}
For all $m, n \in \bbN$, we have
	\begin{equation*}
	 \aps{\calR_{m + n}} \,=\,	\aps{\calR^{(1)}_m} + \aps{\calR^{(2)}_n} - \calE(m,n)  \,,
	\end{equation*}
	where the random variables $\calR^{(1)}_m$ and $\calR^{(2)}_n$ are independent and have  the same law as $\calR_m$ and $\calR_n$, respectively. Moreover, the random variable $\calE (m,n)$ has the same law as $|\calR_m \cap \widetilde{\calR}_n|$.
In particular, $|\calR_m \cap \widetilde{\calR}_n| \le I_{m+n}$.
\end{lemma}
\begin{proof}
	We clearly have
\begin{align*}
		\aps{\calR_{m + n}}
		& \,=\, \aps{\calR_m \cup \calR[m, m + n]}\, =\, \aps{(\calR_m - S_m) \cup (\calR[m, m + n] - S_m)} = \aps{\calR^{(1)}_m \cup \calR^{(2)}_n} \\
		& \,=\, \aps{\calR^{(1)}_m} + \aps{\calR^{(2)}_n} - \aps{\calR^{(1)}_m \cap \calR^{(2)}_n} \,.
	\end{align*}
	The Markov property implies that 
the random variables $\calR^{(1)}_m$ and $\calR^{(2)}_n$ are independent  and that the law of $\calR^{(2)}_n$ is equal to the law of $\calR_n$. By symmetry, $\calR^{(1)}_m$ has the same law as $\calR_m$.
	Evidently, the random variable $\aps{\calR^{(1)}_{m } \cap \calR^{(2)}_{ n}}$ is equal in law to $|\calR_m \cap \widetilde{\calR}_n|$.
The last  inequality $|\calR_m \cap \widetilde{\calR}_n|\le I_{m+n}$ follows by a monotonicity argument, and the proof is finished.
\end{proof}


\begin{proof}[Proof of  Theorem \ref{tm:FCLR_for_the_car}]
To establish the desired result, we again show validity of conditions (i) and (ii) discussed before the proof of Theorem \ref{tm:FCLR_for_the_cap}. We consider the following sequence of random elements in the space $\calD([0,\infty),\R),$
\begin{align*}
X_t^{n}\, =\, \frac{\aps{\calR_{\floor{nt}}} - \bbE\bigl[\aps{\calR_{\floor{nt}}}\bigr]}{\sigma_d \sqrt{n}}\,,\qquad n\ge1\,,
\end{align*}
for a constant $\sigma_d >0$.
By \cite[Theorem 4.5]{LeGall-Rosen}, 
\begin{equation}\label{eq:CLT-X_t^n}
	X_t^{n} \,=\, \frac{\aps{\calR_{\floor{nt}}} - \bbE\bigl[\aps{\calR_{\floor{nt}}}\bigr]}{\sigma_d \sqrt{\floor{nt}}} \cdot \frac{\sqrt{\floor{nt}}}{\sqrt{nt}} \cdot \frac{\sqrt{nt}}{\sqrt{n}} \,\xrightarrow[n\nearrow\infty]{(\text{d})}\, \calN(0, t)\,.
\end{equation}

\noindent
\textit{Condition (i)}.
We choose arbitrary integer $k \ge1$ and fix $0 = t_0 < t_1 < t_2 < \cdots < t_k $. 
 By Cram\'{e}r-Wold theorem \cite[Corollary 5.5]{Kallenberg},
 it suffices to prove that 
\begin{equation}\label{eq:cvg_CW_card}
	\sum_{j = 1}^k \nu_j X_{t_j}^{n} 
	\,\xrightarrow[n\nearrow\infty]{(\text{d})}\, 
	\sum_{j = 1}^k \nu_j B_{t_j}\,, \qquad  (\nu_1, \nu_2, \ldots, \nu_k) \in \bbR^k\,.
\end{equation}

Proceeding in an analogous way as in the proof of Theorem \ref{tm:FCLR_for_the_cap}, and applying Lemma \ref{lm:cardinality_decomposition}, we obtain 
\begin{equation*}
\aps{\calR_{\floor{nt_j}}}  \,=\,	
	\sum_{i = 1}^j \aps{\calR^{(i)}_{\floor{nt_i} -\floor{n t_{i - 1}}} }
	- \sum_{i = 1}^{j - 1}\calE^{(i)}_{\floor{nt_j}}
	\,,\qquad 
	j=1,\ldots ,k\,,
\end{equation*}
where the random variables $\calR^{(i)}_{\floor{nt_i} -\floor{n t_{i - 1}}}$, $i=1,\dots,k$, are independent, and the random variable $\calR^{(i)}_{\floor{nt_i} -\floor{n t_{i - 1}}}$ has the same law as $\calR_{\floor{nt_i} -\floor{n t_{i - 1}}}$. Moreover, 
\begin{align*}
\calE^{(i)}_{\floor{nt_j}} \,=\, | \calR^{(i)}_{\floor{nt_i} -\floor{n t_{i - 1}}}
\cap 
\calR^{(i+1)}_{\floor{nt_j} -\floor{n t_{i }}}
 |\,,
\end{align*}
and it has the same law as $ |\calR_{\floor{nt_i} -\floor{n t_{i - 1}}}
\cap \widetilde{\calR}_{\floor{nt_j} -\floor{n t_{i }}}|$.
Then, we can adapt all the arguments from the proof of Theorem \ref{tm:FCLR_for_the_cap}, and thus we only show how to establish the desired weak convergence. We have the following equality
\begin{equation}\label{conv}
\sum_{j = 1}^k \nu_j X_{t_j}^{n}
	 \,=\, \sum_{i = 1}^k \Bigl(\sum_{j = i}^k \nu_j\Bigr) J^{(i)}_n
	 - \frac{1}{\sigma_d \sqrt{n}} \sum_{j=1}^k \nu_j 	\sum_{i = 1}^{j - 1}\calE^{(i)}_{\floor{nt_j}}
	  + \frac{1}{\sigma_d \sqrt{n}} \sum_{j=1}^k \nu_j 	\sum_{i = 1}^{j - 1}\bbE\bigl[\calE^{(i)}_{\floor{nt_j}}\bigr]\,,
\end{equation}
where
\begin{align*}
J_n^{(i)}\, =\, \frac{\aps{\calR^{(i)}_{\floor{nt_i} -\floor{n t_{i - 1}}} }- \bbE\bigl[\aps{\calR^{(i)}_{\floor{nt_i} -\floor{n t_{i - 1}}} }\bigr]}{\sigma_d \sqrt{n}}\,.
\end{align*}
Using monotonicity argument together with \eqref{eq:trivial}, one can apply \eqref{eq:CLT-X_t^n} to arrive at
\begin{equation*}
	J_n^{(i)} \,\xrightarrow[n\nearrow\infty]{(\text{d})}\, \calN(0, t_i - t_{i - 1})\,, \qquad i=1, 2, \ldots, k\,.
\end{equation*}

We now investigate the convergence of the two last (error) terms in \eqref{conv}.
Recall that $\{S_n\}_{n\ge0}$ and $\{\widetilde S_n\}_{n\ge0}$ are two independent and identically distributed random walks  (defined on the same probability space),  satisfying  (\textbf{A1}) and (\textbf{A2}). Then, in view of  
\cite[Remark after Corollary 3.2]{LeGall-Rosen}, the expectation of the number of their intersection points up to time $n$ admits the following bound
\begin{equation}\label{bound:LeGall}
	\bbE[I_n] \,\le\, C F_d(n)\,,
\end{equation}
where $C>0$ is a constant and $F_d(n)$ is given by
	\[   
	F_d(n) \,=\, 
	\begin{cases}
		1\,, &\quad d/\alpha  > 2\,,\\
		\sum_{k=1}^n k^{-1}\ell (k)^{-d}\,, & \quad d/\alpha = 2\,,\\
		n^2 (b(n))^{-d}\,, & \quad 1 < d/\alpha < 2\,,\\
	\end{cases}
	\]
where $\,b(x)\,$ and $\,\ell(x)\,$ are as in \eqref{eq:shape_of_function_b}. Combining the Markov  inequality with Lemma \ref{lm:cardinality_decomposition}, we obtain
\begin{equation*}
	\bbP \left( n^{-1/2} \calE^{(i)}_{\floor{nt_j}} > \varepsilon \right) 
	\,\le\, 
	\frac{\bbE\bigl[\calE^{(i)}_{\floor{nt_j}}  \bigr]}{\varepsilon \sqrt{n}} 
\,\le\, 	\frac{\bbE\bigl[I_{\floor{nt_j}}  \bigr]}{\varepsilon \sqrt{n}} 
\,	\le\, \frac{CF_d(\floor{nt_j})}{\varepsilon \sqrt{n}}\,.
\end{equation*}
For $d/\alpha \ge 2$, using \cite[Lemma 2.2]{LeGall-Rosen}, we clearly have
\begin{equation*}
	\frac{F_d(\floor{nt_j})}{\sqrt{n}}\, \xrightarrow[n\nearrow\infty]{}\, 0\,.
\end{equation*}
If $3/2 < d /\alpha< 2$, we proceed as follows
\begin{align*}
	\frac{F_d(\floor{nt_j})}{\sqrt{n}}
	& \,=\, \frac{\floor{nt_j}^{2 - d/\alpha} \bigl(\ell(\floor{nt_j})\bigr)^{-d}}{\sqrt{n}} \\
	&\,\le\, \frac{(nt_j)^{2 - d/\alpha} \bigl(\ell(\floor{nt_j})\bigr)^{-d}}{\sqrt{n}} \\ &\,
	 = \,\frac{t_j^{2 - d/\alpha}\bigl(\ell(\floor{nt_j})\bigr)^{-d}}{n^{d/\alpha - 3/2}}\, \xrightarrow[n\nearrow\infty]{}\, 0\,.
\end{align*}
These relations imply that the first error term in \eqref{conv} converges in probability (and whence in distribution) to zero and the second  term is negligible.
Using the fact that $\calR^{(i)}_{\floor{nt_i} -\floor{n t_{i - 1}}}$, $i=1,\dots, k,$ are independent, we obtain that the quantity in \eqref{conv} converges in law to a normal random variable with mean zero and variance $ \sum_{i = 1}^k \bigl(\sum_{j = i}^k \nu_j\bigr)^2 (t_i - t_{i - 1})$.
We finally conclude that 
the finite-dimensional distributions of $\{X^{n}\}_{n\ge1}$ converge weakly to the finite-dimensional distributions of a one-dimensional standard Brownian motion, which means that  condition (i) is satisfied.
\smallskip

\noindent
\textit{Condition (ii)}. 
Let $\{T_n\}_{n \ge1}$ be a bounded sequence of $\{X^{n}\}_{n\ge1}$-stopping times and $\{h_n\}_{n \ge1}\subset[0,\infty)$ a sequence  converging to zero. We prove that
\begin{equation*}
	X^{n}_{T_n + h_n} - X^{n}_{T_n} \,\xrightarrow[n\nearrow\infty]{\bbP}\, 0\,.
\end{equation*}
By the definition, 
\begin{equation*}
	X^{n}_{T_n + h_n} - X^{n}_{T_n} \,=\, \frac{\aps{\calR_{\floor{nT_n + nh_n}}} - \bbE\bigl[\aps{\calR_{\floor{nT_n + nh_n}}}\bigr]}{\sigma_d \sqrt{n}} - \frac{\aps{\calR_{\floor{nT_n}}} - \bbE\bigl[\aps{\calR_{\floor{nT_n}}}\bigr]}{\sigma_d \sqrt{n}}.
\end{equation*}
Combining Lemma \ref{lm:cardinality_decomposition} with  
the following trivial inequality
\begin{equation}\label{eq:trivial2}
	\aps{\calR_{n + 1}} \,=\, \aps{\calR_n \cup \{S_{n + 1}\}} \,\le\, \aps{\calR_n} + 1\,,
\end{equation}
 and with the strong Markov property, we obtain
\begin{align*}
\aps{\calR_{\floor{nT_n + nh_n}}} 
\,\le\, 
\aps{\calR_{\floor{nT_n} + \floor{nh_n}+1}} 
\,\le\, 
\aps{\calR^{(1)}_{\floor{nT_n}}} + \aps{\calR^{(2)}_{\floor{nh_n} + 1}}\,,
\end{align*}
and
\begin{equation*}
\aps{\calR_{\floor{nT_n + nh_n}}} 
\,\geq \,
\aps{\calR_{\floor{nT_n} + \floor{nh_n}}} 
\,\geq \,
	\aps{\calR^{(1)}_{\floor{nT_n}}} + \aps{\calR^{(2)}_{\floor{nh_n}}} 
	- \calE\bigl(\floor{n T_n }, \floor{n h_n}\bigr)\, ,
\end{equation*}
where $\calR^{(1)}_{\floor{nT_n}}$ and $\calR^{(2)}_{\floor{nh_n}}$  are independent and  have the same law as $\calR_{\floor{nT_n}}$ and $\calR_{\floor{nh_n}}$, respectively. Moreover, the random variable $\calE\bigl(\floor{nT_n}, \floor{nh_n)}\bigr)  $ has the same law as $\aps{\calR_{\floor{nT_n}} \cap \widetilde{\calR}_{\floor{nh_n}}}$.

We now show how to bound the sequence $\{X^{n}_{T_n + h_n} - X^{n}_{T_n}\}_{n\ge1}$ from below and above with quantities that converge to zero in probability.
We only sketch the argument for the lower bound, as the second case is similar. We have
\begin{align*}
	X^{n}_{T_n + h_n} - X^{n}_{T_n}
	& \,\ge\, 
	 \frac{\aps{\calR^{(2)}_{\floor{nh_h}}} - \bbE\bigl[\aps{\calR_{\floor{nh_h} + 1}}\bigr]}{\sigma_d \sqrt{n}} 
	- 
	\frac{\calE\bigl(\floor{n T_n }, \floor{n h_n}\bigr) }{\sigma_d \sqrt{n}} \\
	&\, =\, \frac{\aps{\calR^{(2)}_{\floor{nh_h}}} - \bbE\bigl[\aps{\calR_{\floor{nh_h}}}\bigr]}{\sigma_d \sqrt{n}} 
	- \frac{\calE\bigl(\floor{n T_n }, \floor{n h_n}\bigr) }{\sigma_d \sqrt{n}} 
	+ \frac{\bbE\bigl[\aps{\calR_{\floor{nh_h}}}\bigr] - \bbE\bigl[\aps{\calR_{\floor{nh_h} + 1}}\bigr]}{\sigma_d \sqrt{n}} \\
	& \,\ge\, \frac{\aps{\calR^{(2)}_{\floor{nh_h}}} - \bbE\bigl[\aps{\calR_{\floor{nh_h}}}\bigr]}{\sigma_d \sqrt{n}} -
	 \frac{\calE\bigl(\floor{n T_n }, \floor{n h_n}\bigr) }{\sigma_d \sqrt{n}} - 
	 \frac{1}{\sigma_d \sqrt{n}}\,,
\end{align*}
where in the last line we used \eqref{eq:trivial2}.
It remains to prove that 
\begin{equation}\label{eq:LB_goes_to_0_in_P_card}
	\frac{\aps{\calR^{(2)}_{\floor{nh_h}}} - \bbE\bigl[\aps{\calR_{\floor{nh_h}}}\bigr]}{\sigma_d \sqrt{n}} - 
	\frac{\calE\bigl(\floor{n T_n }, \floor{n h_n}\bigr) }{\sigma_d \sqrt{n}}\, \xrightarrow[n \nearrow \infty]{\bbP}\, 0\,.
\end{equation}
We take an  arbitrary $\varepsilon > 0$ and apply  Markov's inequality to arrive at
\begin{align*}
	&\bbP
	 \left( n^{-1/2} \bigl|\bigl( \aps{\calR^{(2)}_{\floor{nh_h}}} - \bbE\bigl[\aps{\calR_{\floor{nh_h}}}\bigr]\bigr) - 
	\calE\bigl(\floor{n T_n }, \floor{n h_n}\bigr)\bigr| > \sigma_d \varepsilon \right) \\
	& \,\le\, \bbP\left( 2 \bigl|\aps{\calR^{(2)}_{\floor{nh_n}}} - \bbE\bigl[\aps{\calR_{\floor{nh_n}}}\bigr] \bigr| > \varepsilon \sigma_d \sqrt{n}\right) 
	+ \bbP \left( 2 \calE\bigl(\floor{n T_n }, \floor{n h_n}\bigr)> \varepsilon \sigma_d \sqrt{n}\right) \\
	& \,\le \,
	\frac{4 \Var(\aps{\calR_{\floor{nh_n}}})}{\varepsilon^2 \sigma_d^2 n} 
	+ \frac{2 \bbE\bigl[ \calE\bigl(\floor{n T_n }, \floor{n h_n}\bigr)\bigr]}{\varepsilon \sigma_d \sqrt{n}}\,.
\end{align*}
By \cite[Theorem 4.4]{LeGall-Rosen}, we know that there exists a constant $C_1 > 0$ such that $\Var(\aps{\calR_n}) \le C_1 n$, for all $n \ge1$. By our assumptions, there is a constant $C$ such that $\sup_{n\ge1}\max\{T_n, h_n\} \leq C$. Moreover, by Lemma \ref{lm:cardinality_decomposition}, $\bbE[ \calE(\floor{n T_n }, \floor{n h_n})] \leq \bbE[I_{\floor{Cn}}]$.
 Hence, by \eqref{bound:LeGall}, we obtain
\begin{align*}
	&\bbP
	 \left( n^{-1/2} \bigl|\bigl( \aps{\calR^{(2)}_{\floor{nh_h}}} - \bbE\bigl[\aps{\calR_{\floor{nh_h}}}\bigr]\bigr) - 
	\calE\bigl(\floor{n T_n }, \floor{n h_n}\bigr)\bigr| > \sigma_d \varepsilon \right) \\
	& \,\le\, \frac{4C_1 n h_n}{\varepsilon^2 \sigma_d^2 n} + \frac{2C_2 F_d(\floor{Cn})}{\varepsilon \sigma_d \sqrt{n}}  \,.
\end{align*}
Since $h_n \to 0$ and $F_d(\floor{Cn}) / \sqrt{n} \to 0$ as $n \nearrow \infty$, we conclude \eqref{eq:LB_goes_to_0_in_P_card},
and  the proof is finished.  
\end{proof}

\bibliographystyle{abbrv}
\bibliography{FCLT}

\end{document}